\documentclass[smallextended]{svjour3}

\usepackage{amsmath,amssymb}
\usepackage[misc]{ifsym}
\newcommand{\envelope}{(\kern1pt\Letter\kern1pt)}

\usepackage[numbers,sort&compress]{natbib}
\usepackage{tabularx}

\newcommand{\setof}[1]{\left\{{#1}\right\}}
\newcommand{\cset}[2]{\setof{#1\,:\,#2}}
\newcommand{\ivcc}[2]{\left[#1,#2\right]}
\newcommand{\ivco}[2]{\left[#1,#2\right)}
\newcommand{\ivoc}[2]{\left(#1,#2\right]}

\newcommand{\abs}[1]{\left|#1\right|}
\newcommand{\norm}[1]{\left\Vert#1\right\Vert}
\newcommand{\normi}[1]{\left\Vert#1\right\Vert_\infty}
\newcommand{\normop}[1]{\left\Vert#1\right\Vert_{op}}

\newcommand{\spacecf}{C(\ivcc{0}{1})}
\newcommand{\spacedf}[1]{C^{{#1}}(\ivcc{0}{1})}
\newcommand{\spacepp}{\mathcal{S}(\Delta_n, k)}
\newcommand{\spaceppk}[1]{\mathcal{S}(\Delta_n, {#1})}

\newcommand{\ballo}[2]{B({#1},{#2})}
\newcommand{\ballc}[2]{\overline{\ballo{#1}{#2}}}
\newcommand{\uballo}{\ballo{0}{1}}
\newcommand{\uballc}{\ballc{0}{1}}

\newcommand{\Nn}{{\mathbb N}}
\newcommand{\Rr}{{\mathbb R}}
\newcommand{\Cc}{{\mathbb C}}

\newcommand{\Sdk}{S_{{\Delta_n}, k}\,}

\newcommand{\Sdkp}[1]{S^{#1}_{{\Delta_n}, k}\,}
\newcommand{\Sdkpm}{\Sdkp{m}}

\newcommand{\bspk}[1]{N_{{#1}, k}}
\newcommand{\bspkm}[2]{N_{{#1}, {k-#2}}}

\newcommand{\meshmin}{\delta_{\mathrm{min}}}
\newcommand{\meshmax}{\delta_{\mathrm{max}}}

\allowdisplaybreaks

\DeclareMathOperator{\im}{Im}

\journalname{Constructive Approximation}

\begin{document}

\title{Lower bounds for the approximation with variation-diminishing splines}
\author{J. Nagler \and P. Cerejeiras \and B. Forster}

\institute{
  J. Nagler \envelope \and B. Forster \at 
  Fakult\"at f\"ur Informatik und Mathematik, Universit\"at Passau, Germany\\
  \email{johannes.nagler@uni-passau.de}\\
  \email{brigitte.forster@uni-passau.de}
  \and
  P. Cerejeiras \at 
  Department of Mathematics, University of Aveiro, Portugal\\
  \email{pceres@ua.pt}
}

\date{Preprint, February 10, 2014}

\maketitle

\begin{abstract}
  We prove lower bounds for the approximation error of the variation-diminishing 
  Schoenberg operator on the interval $\ivcc{0}{1}$ in terms of classical moduli
  of smoothness depending on the degree of the spline basis using a
  functional analysis based framework.  Thereby, we characterize the
  spectrum of the Schoenberg operator and investigate the asymptotic
  behavior of its iterates. Finally, we prove the equivalence between
  the approximation error and the
  classical second order modulus of smoothness as an improved version of an open conjecture from 2002.
\end{abstract}

\section{Introduction}
Schoenberg introduced the variation-diminishing splines already in
1946 as part of a natural extension of the classical Bernstein
polynomials. Even though, his ideas have first been published
20 years later in the well known article of \citet{Curry:1966}. Since
then, they have attracted the interest of the scientific community due
to their good properties and vast range of applications. An
comprehensive overview on the theory of splines can be found in the
books of \citet{deBoor:1987, nuernberger1989}, and \citet{schumaker2007}.

In 2002, L. Beutel and her coauthors gave in the article \cite{Beutel:2002} a short survey
on the history, developments and contributions in this area and
investigated quantitative direct approximation inequalities for the Schoenberg operator. 
More importantly, the authors stated an interesting conjecture regarding
the equivalence of the approximation error of the Schoenberg operator
on $\ivcc{0}{1}$ and the second order Ditzian-Totik modulus of
smoothness.

We show an even stronger result, namely the equivalence with the
classical second order modulus of smoothness. Thereby, we first
characterize the asymptotic behavior of the iterates of the Schoenberg
operator. Afterwards, we use this result in order to prove a lower
bound of the approximation error with respect to the second order
modulus of smoothness.

The convergence of the iterates of the Schoenberg operator to the
operator of linear interpolation at the endpoints of the interval
$\ivcc{0}{1}$ can be also seen by the results of the article of
\citet{gavrea2011}. However, while their methods ensure the uniform
convergence of those iterates, they do not give the rate of
convergence in which in fact we are interested.  Therefore, our
approach uses an earlier result of C. Badea \cite{Badea:2009}, where the
asymptotic behavior of the iterates is characterized by their spectral
properties. Moreover, these results provide a simple and elegant
generalization of the results of \citet{nagler2013} to the non-uniform
case by using a functional analysis based framework, which we believe
is beautiful on its own.

\subsection{The Schoenberg operator}
Let $n > 0, k > 0$ be integers and let $\Delta_n = \{x_j\}_{j=-k}^{n+k}$ be a partition of $\ivcc{0}{1}$ satisfying
\begin{equation*}
  x_{-k} = \cdots = x_0 = 0 < x_1 < \ldots < x_n = \cdots = x_{n+k} = 1.
\end{equation*}
Throughout this paper, we will consider the Banach space $\spacecf$, the space of real-valued continuous functions on the interval $\ivcc{0}{1}$ endowed with the uniform norm $\normi{\cdot}$,
\begin{equation*}
  \norm{f}_\infty = \sup\cset{\abs{f(x)}}{x \in \ivcc{0}{1}}, \qquad f\in \spacecf.
\end{equation*}

The variation-diminishing spline operator $\Sdk:\spacecf \to \spacecf$ 
of degree $k$ with respect to the knot sequence $\Delta_n$ is defined for $f \in \spacecf$ by
\begin{align*}
  \Sdk f(x) &= \sum_{j=-k}^{n-1}f(\xi_{j,k})\bspk{j}(x),\quad 0 \leq x < 1,\\
  \Sdk f(1) &= \lim_{y \nearrow 1} \Sdk f(y)
\end{align*}
with the so called Greville nodes 
\begin{equation*}
  \xi_{j,k} := \frac{x_{j+1} + \cdots + x_{j+k}}{k},\quad -k \leq j \leq n-1,
\end{equation*}
and the normalized B-splines 
\begin{equation*}
  \bspk{j}(x) := (x_{j+k+1} - x_j)[x_j,\ldots,x_{j+k+1}](\cdot - x)_+^k.
\end{equation*}
Here, the divided difference $[x_j,\ldots,x_{j+k+1}]f$ for $f \in \spacecf$ is defined to be the coefficient
of $x^k$ in the unique polynomial of degree $k$ or less that interpolates the function $f$ 
at the points $x_j,\ldots,x_{j+k+1}$. We denote by $x^k_+$ the truncated power function of degree $k$,
defined for $x \in \Rr$ by
\begin{equation*}
  x^k_{+} =
  \begin{cases}
    x^k, & \text{for }x > 0,\\
    0, & \text{for } x \leq 0.
  \end{cases}
\end{equation*}

The operator $\Sdk$ was first discussed by Schoenberg and Curry in 1966
 as a generalization of the Bernstein operator see, e.g., \cite{Curry:1966,Marsden:1970}.
The normalized B-splines form a partition of unity
\begin{equation}
  \label{eq:partition_unity}
  \sum_{j=-k}^{n-1}\bspk{j}(x) = 1,
\end{equation}
and the Schoenberg operator reproduces linear functions, i.e.,
\begin{equation}
  \label{eq:reproduce_linear}
  \sum_{j=-k}^{n-1}\xi_{j,k}\bspk{j}(x) = x,
\end{equation}
due to the chosen Greville nodes. A comprehensive overview of
direct approximation inequalities for this operator can be found in
\cite{Beutel:2002}.

\subsection{Notation}
We denote the space of bounded linear operators on $\spacecf$ by $\mathcal{B}(\spacecf)$ equipped with the usual operator norm $\normop{\cdot}$. With $I$ we denote the identity operator on $\mathcal{B}(\spacecf)$. For $T \in \mathcal{B}(\spacecf)$, we denote by $\sigma(T)$ the spectrum of $T$, 
\begin{equation*}
  \sigma(T) = \cset{\lambda \in \Cc}{T - \lambda I \text{ is not invertible}}.
\end{equation*}
By $\sigma_p(T)$, we denote the point spectrum of $T$, 
\begin{equation*}
  \sigma_p(T) = \cset{\lambda \in \Cc}{T - \lambda I \text{ is not one-to-one}},
\end{equation*}
which contains all the eigenvalues of $T$.
We denote by $\spacepp$ the spline space of degree $k$ with respect to the knot sequence $\Delta_n$,
\begin{equation*}
  \spacepp = \cset{\sum_{j=-k}^{n-1}c_j \bspk{j}}{c_j \in \Rr,\ j \in \setof{-k,\ldots,n-1}} \subset \spacedf{k-1}.
\end{equation*}
Since $\spacepp$ is an $n+k$-dimensional subspace of $\spacecf$,
$\spacepp$ is a Banach space with the inherited norm $\normi{\cdot}$.
For more information on spline spaces see, e.g., \cite{deBoor:1987, nuernberger1989, schumaker2007}.
The open ball of radius $r > 0$ at the point $z \in \Cc$ in the
complex plane will be denoted by $\ballo{z}{r} := \cset{\lambda \in
  \Cc}{\abs{\lambda - z} < r}$ and its closure by $\ballc{z}{r}$.

\section{The spectrum of the Schoenberg operator}
We investigate some basic properties of the Schoenberg operator needed in order to prove our main results, and that are of interest of their own.
The following fact can, e.g., be found in \cite{deBoor:1973}.
\begin{theorem}
  \label{thm:bounded}
  The Schoenberg operator $\Sdk:\spacecf \to \spacecf$ is bounded and $\normop{\Sdk} = 1$.
\end{theorem}
\begin{proof}
  Let $f \in \spacecf$ with $\normi{f} = 1$. Then
  \begin{equation*}
    \normi{\Sdk f} = \normi{\sum_{j=-k}^{n-1}f(\xi_{j,k})\bspk{j}(x)}
    \leq \normi{f}\cdot \normi{\sum_{j=-k}^{n-1}\bspk{j}(x)} = 1,
  \end{equation*}
  because of property \eqref{eq:partition_unity}. Therefore, $\norm{\Sdk} \leq 1$.
  By considering now the constant function $1 \in \spacecf$, we get
  $\norm{\Sdk 1}_\infty = 1$. Hence, the operator has norm $1$, $\normop{\Sdk} = 1$.
\end{proof}

Due to the finite-dimensional image of $\Sdk$, we can directly obtain the compactness of the Schoenberg operator.
\begin{theorem}
  \label{thm:compact}
  The Schoenberg operator $\Sdk:\spacecf \to \spacecf$ is compact and as such
  $\im(\Sdk - I)$ is closed. Besides, $1$ is not a cluster point of the spectrum $\sigma(\Sdk)$.
\end{theorem}
\begin{proof}
  From Theorem \ref{thm:bounded} it follows that the operator is bounded with $\normop{\Sdk} = 1$ and maps continuous
  functions to the spline space $\spacepp$. Therefore,  the operator has finite rank and finite rank operators are compact.
  For compact operators $\im(T - I)$ is closed and $0$ is the only possible cluster point of $\sigma(\Sdk)$, see \cite{rudin1991}, Theorem~4.25.
\end{proof}

The main result of this section is the following:
\begin{theorem}
  The spectrum of the Schoenberg operator consists only of the point spectrum and
  \begin{equation*}
    \sigma(\Sdk) \subset \uballo \cup \setof{1}.
  \end{equation*}
\end{theorem}
\begin{proof}
  Since $\normop{\Sdk} = 1$, for $\lambda \in \sigma(\Sdk)$ the inequality
  \begin{equation*}
    \abs{\lambda} \leq \normop{\Sdk} = 1
  \end{equation*}
  holds.  Therefore, $\sigma(\Sdk) \subset \uballc$. 

  In the following, we show that $\sigma(\Sdk) \subset \uballo \cup \setof{1}$, i.e., if $\lambda \in \sigma(\Sdk)$ with 
  $\abs{\lambda} = 1$ then $\lambda = 1$. First, we prove that $0 \in \sigma_p(\Sdk)$. Then, we will show that $1 \in \sigma_p(\Sdk)$.
  Finally, we consider eigenvalues $\lambda \in \sigma_p(\Sdk) \setminus \setof{0,1}$ and we show that then $\abs{\lambda} < 1$ holds.
  
  \begin{description}
    \item \emph{Step 1}: We show that $0 \in \sigma_p(\Sdk)$. Let $f \in \spacecf$ a
    function, such that
    \begin{equation*}
      f(\xi_j) = 0 \qquad \text{for all } j \in \setof{-k, \ldots, n-1} 
    \end{equation*}
    and such that there exists $x \in \ivcc{0}{1} \setminus \cset{\xi_j}{j \in
      \setof{-k,\ldots,n-1}}$ with $f(x) \neq 0$.  For example,
    consider the polynomial $f(x) =
    \prod_{i=-k}^{n-1}(x-\xi_i)$. Clearly, $f \in \spacecf$ and we
    obtain $\Sdk f = 0 \cdot f = 0$, because for all $x \in \ivcc{0}{1}$
    \begin{equation*}
      \Sdk f(x) = \sum_{j=-k}^{n-1}\left[\prod_{i=-k}^{n-1}(\xi_j-\xi_i)\right]\bspk{j}(x) = 0.
    \end{equation*}
    For compact operators, it is known that every $\lambda \neq 0$ in the spectrum
    is contained in the point spectrum of the operator. This classical result 
    is stated, e.g., in  Theorem~4.25 by \citet{rudin1991}.

    As $0 \in \sigma_p(\Sdk)$, it follows that
    \begin{equation*}
      \sigma(\Sdk) = \sigma_p(\Sdk).
    \end{equation*}

    \item \emph{Step 2}: We have $1 \in \sigma(\Sdk)$, because of the properties
    \eqref{eq:partition_unity} and \eqref{eq:reproduce_linear} and the
    function $f(x) = 1$ and $f(x) = x$ are eigenfunctions of $\Sdk$
    corresponding to the eigenvalue $1$.

    \item \emph{Step 3}: Now we prove that for all the other eigenvalues $\lambda \in
    \sigma(\Sdk)$, we have
    \begin{equation*}
      \abs{\lambda} < 1.
    \end{equation*}
    Let $\lambda \in \sigma(\Sdk) \setminus \setof{0}$. As the
    operator maps continuous functions to the spline space, the
    eigenfunctions have to be spline functions as well.  Let $s \in
    \spacepp$, $s = \sum_{j=-k}^{n-1}c_j \bspk{j}$, be such an
    eigenfunction for the eigenvalue $\lambda \in \Cc$.  Then
    \begin{alignat*}{2}
      && & \Sdk s = \lambda s\\
      &&\Longleftrightarrow\qquad &
      \sum_{i=-k}^{n-1}\sum_{j=-k}^{n-1}c_j\bspk{j}(\xi_{i})\bspk{i}(x)
      = \lambda \sum_{i=-k}^{n-1}c_i \bspk{i}(x)\\
      &&\Longleftrightarrow\qquad &
      \sum_{i=-k}^{n-1}\left[\sum_{j=-k}^{n-1}c_j\bspk{j}(\xi_{i}) -
        \lambda c_i\right]\bspk{i}(x)
      = 0\\
      &&\Longleftrightarrow\qquad &
      \sum_{j=-k}^{n-1}c_j\bspk{j}(\xi_{i}) = \lambda c_i.
    \end{alignat*}
    Thus, $\lambda \neq 0$ is an eigenvalue of the operator $\Sdk$, if
    and only if $\lambda$ is an eigenvalue of the matrix $N \in
    \Rr^{(n+k) \times (n+k)}$,
    \begin{equation*}
      N =
      \begin{pmatrix}
        N_{-k}(\xi_{-k}) & N_{1-k}(\xi_{-k}) & \cdots & N_{n-1}(\xi_{-k})\\
        N_{-k}(\xi_{1-k}) & N_{1-k}(\xi_{1-k}) & \cdots & N_{n-1}(\xi_{1-k})\\
        \vdots & & &\\
        N_{-k}(\xi_{n-1}) & N_{1-k}(\xi_{n-1}) & \cdots &
        N_{n-1}(\xi_{n-1})
      \end{pmatrix}.
    \end{equation*}
    This matrix $N$ is non-negative as $N_j \geq 0$ and every row sums
    up to one because of property \eqref{eq:partition_unity}.  By the
    Theorem of Gershgorin \cite{Gershgorin:1931}, we have that
    the eigenvalues are contained in the union of circles
    \begin{equation*}
      \lambda \in \bigcup_{j=-k}^{n-1} D_j,
    \end{equation*}
    with 
    \begin{equation*}
      D_j = \cset{\lambda \in \Cc}{\abs{\lambda - \bspk{j}(\xi_j)}
        \leq \sum_{i=-k, i\neq j}^{n-1} \bspk{j}(\xi_i)}.
    \end{equation*}
    Using property \eqref{eq:partition_unity}, it follows that
    \begin{equation*}
      \bigcup_{j=-k}^{n-1} D_j \cap \cset{\lambda \in \Cc}{\abs{\lambda} = 1}   = \setof{1}.
    \end{equation*}  
    Finally, we obtain $\sigma_p(\Sdk) = \sigma(\Sdk) \subset \uballo \cup \setof{1}$.
  \end{description}
\end{proof}

\section{Main Results}
We investigate the iterates $\Sdkpm$ of the Schoenberg operator for $m \to \infty$ and prove a lower bound.
\subsection{The limit of the  iterates of the Schoenberg operator}
We show that the iterates of the Schoenberg operator converge in the limit to a linear operator $L$.
Concretely, we define the iterates by $\Sdkp{0} = I$ and for $m \in \Nn$ by
\begin{equation*}
  \Sdkpm f(x) = \Sdkp{m-1}(\Sdk f)(x).
\end{equation*}
We will show
\begin{equation*}
\lim_{m \to \infty}\normop{\Sdkpm - L} = 0,
\end{equation*}
where $L$ is defined for $f \in \spacecf$ by 
\begin{equation*}
(Lf)(x) = f(0) + \big(f(1) - f(0)\big)\cdot x.
\end{equation*}

In \cite{Badea:2009} it has been shown that operators of a certain structure converge to this linear operator $L$.
In fact, the Schoenberg operator $\Sdk: \spacecf \to \spacecf$ fulfills the following required properties:
\begin{itemize}
\item The operator $\Sdk$ is bounded and $\im(\Sdk - I)$ is closed,
\item $\ker(\Sdk - I) = \mathrm{span}(1, x)$, i.e., the Schoenberg operator reproduces constant and linear functions,
\item $\Sdk f(0) = f(0)$ and $\Sdk f(1) = f(1)$ for every $f \in \spacecf$, i.e., the Schoenberg operator interpolates start and end points,
\item $\sigma(\Sdk) \subset \uballo \cup \setof{1}$, and finally,
\item $1$ is not a cluster point\footnote{This condition was not
    contained in the work of \citet{Badea:2009}, but was needed in our
    proof for the convergence of the iterates. To the best of our
    knowledge it is an open question whether this condition is also
    necessary in the proof for general continuous linear operators as stated in
    Theorem~2.1 and Theorem~2.2 in \cite{Badea:2009}. Anyhow, both
    Theorems hold true for compact operators as in our case.} of
  $\sigma(\Sdk)$, i.e,
  \begin{equation*}
    \sup\cset{\abs{\lambda}}{\lambda \in \sigma(\Sdk) \setminus \setof{1}} < 1.
  \end{equation*}
\end{itemize}
All these properties were deduced in the previous section. We can conclude:
\begin{theorem}
  \label{thm:iterates}
  With $\gamma_{{\Delta_n}, k} := \sup\cset{\lambda \in \Cc}{\lambda \in \sigma(\Sdk) \setminus \setof{1}}$, we obtain
  \begin{equation*}
    \normop{\Sdkpm - L} \leq C \cdot \gamma_{{\Delta_n}, k}^m
  \end{equation*}  
  for some suitable constant $1 \leq C \leq 1/(\gamma_{{\Delta_n}, k})$ and therefore,
  \begin{equation*}
    \lim_{m \to \infty} \normop{\Sdkpm - L} = 0.
  \end{equation*}
\end{theorem}
\begin{proof}
  The result follows now immediately from Theorem 2.1 in \cite{Badea:2009} using the above mentioned properties of $\Sdk$.
\end{proof}

\subsection{A lower bound of the Schoenberg operator}
In this section, we show that for $r \in \Nn$, $r \geq 2$, $0 < t \leq \frac1{r}$ and $k > r$, there exists a constant $M > 0$, such that
\begin{equation*}
  M \cdot \omega_r(f,t) \leq \normi{f - \Sdk f}.
\end{equation*}
Here the $r$-th modulus of smoothness $\omega_r:  \spacecf \times \ivoc{0}{\frac{1}{r}} \to \ivco{0}{\infty}$ is defined by
\begin{equation*}
  \omega_r(f, t) := \sup_{0<h<t}\sup\cset{|\Delta_h^rf(x)|}{x \in \ivcc{0}{1-rh}},
\end{equation*}
with the forward difference operator
\begin{equation*}
  \Delta_h^kf(x) = \sum_{l=0}^{r}(-1)^{r-l} \binom{r}{l}f(x + lh).
\end{equation*}

The $r$-th modulus of smoothness satisfies the following properties \cite{zygmund2002,timan1994}:
\begin{lemma} Let $0 < t \leq \frac1{r}$ be fixed.
  \begin{enumerate}
  \item For $f_1, f_2 \in \spacecf$, the triangle inequality holds,
    \begin{equation}
      \label{eq:mos_triangle}
      \omega_r(f_1+f_2, t) \leq \omega_r(f_1, t) + \omega_r(f_2, t).
    \end{equation}
  \item If $f \in \spacecf$, then
    \begin{equation}
      \label{eq:mos_cf}
      \omega_r(f, t) \leq 2^r \normi{f}.
    \end{equation}
  \item If $f \in \spacedf{r}$, then
    \begin{equation}
      \label{eq:mos_cd}
      \omega_r(f, t) \leq t^r \normi{D^rf}.
    \end{equation}
  \end{enumerate}
\end{lemma}

Note that for $k > r$ the spline space $\spacepp \subset \spacedf{r}$, because $\Sdk f \in \spacedf{k-1}$. 
Hence, using inequalities \eqref{eq:mos_triangle} -- \eqref{eq:mos_cd}, we obtain
\begin{equation}
  \label{eq:mos_inequality}
  \omega_r(f, t) \leq  2^r \normi{f-\Sdk f}+t^r\normi{D^r \Sdk f}.
\end{equation}
See also \cite{Butzer:1967,Johnen:1976} for the equivalence to the $K$-functional.
In the following we will estimate the last term with respect to the approximation error
$\normi{\Sdk f - f}$. To this end, we consider the minimal mesh length $\meshmin$,
\begin{equation*}
  \meshmin := \min\cset{(x_{j+1} - x_{j})}{j \in \setof{-k, \ldots, n-2}}.
\end{equation*}

\begin{lemma}
  \label{lem:bounded_diff_op}
  The differential operator $D: \spacepp \to \spaceppk{k-1}$ is bounded
  with $\normop{D} \leq (2/\meshmin)d_k$, where $d_k > 0$ is a constant depending only on $k$.
\end{lemma}
\begin{proof}
  Let $s \in \mathcal{S}(\Delta_n, k)$, $s(x) = \sum_{j=-k}^{n-1}c_j
  \bspk{j}{x}$, with $\norm{s}_\infty = 1$. According to \cite{Marsden:1970}, we can calculate the
  derivative by
  \begin{equation*}
    Ds(x) = \sum_{j=1-k}^{n-1}\frac{c_j - c_{j-1}}{\xi_{j} - \xi_{j-1}} \bspkm{j}{1}(x).
  \end{equation*}
  Then we obtain with the triangle inequality
  \begin{align*}
    \normi{D s} &=  \normi{\sum_{j=1-k}^{n-1}\frac{c_j - c_{j-1}}{\xi_{j} - \xi_{j-1}} \bspkm{j}{1}}\\
               &\leq  \frac{\normi{c} + \normi{c}}{\meshmin} \cdot \normi{ \sum_{j=1-k}^{n-1}\bspkm{j}{1}},
  \end{align*}
  where
  \begin{equation}
    \label{eq:norm_coeff}
    \normi{c} = \max\cset{\abs{c_j}}{j \in \setof{-k, \ldots, n-1}}.
  \end{equation}
  According to \cite{deBoor:1973}, there exists $d_k > 0$ depending only on $k$, such that
  \begin{equation}
    \label{eq:stability_spline}
    d_k^{-1} \norm{c}_\infty \leq \norm{\sum_{j=-k}^{n-1}c_j\bspk{j}}_\infty \leq \norm{c}_\infty.
   \end{equation}
   Rewriting the first inequality yields $\norm{c}_\infty \leq d_k$, because $\normi{s} = 1$.
   Now we use the partition of the unity \eqref{eq:partition_unity} to derive the estimate
  \begin{align*}
    \normi{D s} 
               &\leq \frac{2}{\meshmin}d_k.
  \end{align*}
  Taking the supremum of all $s \in \spacepp$ with $\normi{s} = 1$ yields the result.
\end{proof}

\begin{corollary}
  \label{cor:differential_op}
  For $l < k$, the differential operators $D^l: \spacepp \to \spaceppk{k-l}$ are bounded and 
  \begin{equation*}
    \normop{D^l} \leq (2/\meshmin)^{l}d_k. 
  \end{equation*}
\end{corollary}
\begin{remark}
  The asymptotic behaviour of the constant $d_k$ in Lemma~\ref{lem:bounded_diff_op} is already
  characterized quite well in the literature. C. de Boor conjectured that
  \begin{equation*}
    d_k \sim 2^k
  \end{equation*}
  holds for all $k > 0$. He showed in \cite{deBoor1990} with numerical computations that
  \begin{equation*}
    d_k \leq c \cdot 2^k
  \end{equation*}
  for some constant $c > 0$. 
  In \cite{lyche1987},  T. Lyche proved the lower bound
  \begin{equation*}
    2^{-3/2}\frac{k-1}{k} \cdot 2^k \leq d_k.
  \end{equation*}
  C. de Boor's conjecture was confirmed in the article \cite{scherer1999} of \citeauthor{scherer1999}
  up to a polynomial factor. 
  There the authors showed that the inequality
  \begin{equation*}
    d_k \leq k \cdot 2^k
  \end{equation*}
  holds for all $k > 0$. In our interest is the relation $d_k \to \infty$ if $k$ tends to infinity.
\end{remark}

Now we are able to estimate $\normi{D^r\Sdk f}$ in terms of the approximation error $\normi{f- \Sdk f}$. 
\begin{theorem}
  \label{thm:lower_bound}
  Let $f \in \spacecf$, $r \geq 2$, $k > r$ and $0 < t \leq \frac1{r}$. Then there exists $M > 0$, such that
\begin{equation*}
  M \cdot \omega_r(f,t) \leq \normi{f - \Sdk f}.
\end{equation*}
\end{theorem}

\begin{proof}
  We derive
  \begin{align*}
    \normi{D^r \Sdk f} &=  \normi{D^r \Sdk f- D^r \Sdkp{2} f + D^r \Sdkp{2} f -D^r \Sdkp{3} f + \ldots}\\
    &\leq \sum_{m=1}^\infty \normi{D^r \Sdkpm (f-\Sdk f)}\\
    &\leq \normi{f - \Sdk f} \sum_{m=1}^\infty \normop{D^r\Sdkp{m}}\\
    &= \normi{f - \Sdk f} \sum_{m=1}^\infty \normop{D^r(\Sdkp{m} - L + L)}\\
    &=  \normi{f - \Sdk f} \sum_{m=1}^\infty \normop{D^r(\Sdkp{m} - L)},\\
    \intertext{as $D^r$ annihilates linear functions and therefore,
      $D^rL = 0$. Then we obtain using Theorem \ref{thm:iterates}
      and Corollary \ref{cor:differential_op}}
    \norm{D^r \Sdk f} &\leq \normi{f - \Sdk f} \normop{D^r}\sum_{m=1}^\infty \normop{\Sdkp{m} - L}\\
    &\leq \normi{f - \Sdk f} \normop{D^r}\sum_{m=1}^\infty C\gamma_{{\Delta_n}, k}^m\\
    &\leq \normi{f - \Sdk f} \normop{D^r} \frac{C\gamma_{{\Delta_n}, k}}{1-\gamma_{{\Delta_n}, k}}\\
    &\leq \frac{2^r\gamma_{{\Delta_n}, k} d_k C}{\meshmin^r(1-\gamma_{{\Delta_n}, k})}\, \normi{f-\Sdk f}.
  \end{align*}
  As $C \leq 1/\gamma_{{\Delta_n},k}$, we get
  \begin{equation*}
    \normi{D^r \Sdk f} \leq \frac{2^r d_k}{\meshmin^r(1-\gamma_{{\Delta_n}, k})}\, \normi{f-\Sdk f}.
  \end{equation*}
  Applying inequality \eqref{eq:mos_inequality} for $0 < t \leq \frac1{r}$ yields
  \begin{equation*}
    \omega_r(f,t) \leq 2^r\left(1 + \frac{d_k}{\meshmin^r(1-\gamma_{{\Delta_n}, k})}t^r \right) \cdot \normi{f-\Sdk f}.
  \end{equation*}
\end{proof}

\begin{corollary}
  \label{cor:lower_bound}
  For all $f \in \spacecf$ and $r \geq 2$, the approximation error cannot be better than
  \begin{equation*}
    \frac1{2^{r + 1}}\omega_r(f, \delta) \leq \normi{f-\Sdk f},
  \end{equation*}
  where
  \begin{equation*}
    \delta = \meshmin\cdot\left(\frac{1-\gamma_{{\Delta_n}, k}}{d_k}\right)^{1/r}
  \end{equation*}
  given a fixed grid $\Delta_n$ and the degree $k$ of the spline approximation. For $n \to \infty$ and
 $k \to \infty$, we have $\delta \to 0$ as $\meshmin \to 0$ and $d_k \to \infty$ respectively.
\end{corollary}

\begin{remark}
  Note that our result can also be generalized to the Ditzian-Totik modulus of smoothness, as
  a similar equivalence to a weighted K-functional holds.
\end{remark}

\begin{corollary}
   For $0 < t \leq \frac1{2}$ and $k > 2$, we obtain the equivalence 
  \begin{equation*}
    \omega_2(f, t) \sim \normi{f - \Sdk f}
  \end{equation*}
  in the sense that there exist constants $M_1, M_2 > 0$ independent of $f$ and only depending on $\Delta_n$ and $k$ such that
  \begin{equation*}
    M_1 \cdot \omega_r(f, t) \leq \normi{f - \Sdk f} \leq M_2 \cdot \omega_r(f, t).
  \end{equation*}
\end{corollary}
\begin{proof}
  We apply Corollary~\ref{cor:lower_bound} to get the lower inequality
  \begin{equation*}
    \frac{1}{8} \cdot  \omega_2(f, \sqrt{\frac{(1-\gamma_{{\Delta_n}, k})\cdot \meshmin^2}{d_k}}) \leq \normi{f-\Sdk f}.
  \end{equation*}  
  We use the inequality
  \begin{equation*}
    \normi{f - \Sdk f} \leq \frac{3}{2}\cdot \omega_2(f, 
       \sqrt{ \min\left\{ \frac{1}{2k}, \frac{(k+1)\cdot \meshmax^2}{12} \right\} })
  \end{equation*}
  from \cite{Beutel:2002}, Theorem 6, 
  to obtain the upper bound. Here
  \begin{equation*}
    \meshmax := \max\cset{(x_{j+1} - x_j)}{j \in \setof{-k,\ldots,n-1}}.
  \end{equation*}  
\end{proof}

Finally, there is still one open question to answer. By definition of
the constants, we have $d_k \to \infty$ for $k \to \infty$ and
$\meshmin \to 0$ for $n \to \infty$. The question is whether the second largest eigenvalues of the
operator can speed up the convergence in Corollary~\ref{cor:lower_bound}. 
As far as we know, the eigenvalues and eigenfunctions of the Schoenberg operator are still unknown. 
We conclude the article with the following conjecture that characterizes the behavior of the
second largest eigenvalue of the Schoenberg operator.

\begin{conjecture}
  Let $k>0$ be fixed. Then
  \begin{equation*}
    \gamma_{\Delta_n,k} \to 1,\quad\text{for } n \to \infty.
  \end{equation*}
  Let $n>0$ be fixed. Then
  \begin{equation*}
    \gamma_{\Delta_n,k} \to 1,\quad\text{for } k \to \infty.
  \end{equation*}
\end{conjecture}

\begin{acknowledgements}
  The work of the first and the last author was partially supported by
  the DAAD grant ``Conformal monogenic frames for image analysis'',
  Project Id 54367931 respectively 57030516, PPP Programme, awarded to
  Brigitte Forster.  The work of the second author was partially
  supported by Portuguese funds through the CIDMA - Center for
  Research and Development in Mathematics and Applications, and the
  Portuguese Foundation for Science and Technology (FCT-Funda\c c\~ao
  para a Ci\^encia e a Tecnologia), within project
  PEst-OE/MAT/UI4106/2014.
\end{acknowledgements}

\bibliographystyle{plainnat}
\bibliography{references}
\nocite{*}
\end{document}